\documentclass[11pt,letterpaper]{amsart}

\usepackage{mathrsfs}
\usepackage[OT2, T1]{fontenc}
\usepackage{url}
\usepackage{amsmath}
\usepackage{array}
\usepackage{graphicx}
\usepackage{amsfonts}
\usepackage{amssymb}
\usepackage{amstext}
\usepackage{amsthm}
\usepackage{hyperref}
\usepackage{colonequals}
\usepackage{enumitem}
\usepackage[alphabetic,lite]{amsrefs}
\usepackage{cleveref}
\usepackage[all,cmtip]{xy}

\numberwithin{equation}{subsection}

\newtheorem{theorem}{Theorem}[section]
\newtheorem{lemma}[theorem]{Lemma}
\newtheorem{proposition}[theorem]{Proposition}

\theoremstyle{definition}
\newtheorem{defn}{Definition}

\theoremstyle{remark}

\newcommand{\pdiv}{\mathscr{G}}
\newcommand{\lift}{\tilde{\pdiv}}
\newcommand{\liftphi}{\tilde{\phi}}

\newcommand{\F}{\mathbb{F}}

\newcommand{\Q}{\mathbb{Q}}
\newcommand{\Qbar}{\overline{\Q}}

\newcommand{\Z}{\mathbb{Z}}
\newcommand{\ol}[1]{\overline{#1}}
\newcommand{\Fpbar}{\overline{\mathbb{F}}_p}
\newcommand{\Pone}{\mathbb{P}^1}

\newcommand{\bbP}{\mathbf{P}}

\newcommand{\cA}{\mathcal{A}}

\newcommand{\cD}{\mathcal{D}}

\newcommand{\cL}{\mathcal{L}}

\newcommand{\cO}{\mathcal{O}}

\newcommand{\fm}{\mathfrak{m}}

\newcommand{\Ag}{\mathcal{A}_g}

\newcommand{\Agr}{\Ag[r]}
\newcommand{\Agrtor}{\Agr^{\tor}}

\newcommand{\Mg}{\mathcal{M}_g}

\newcommand{\ra}{\rightarrow}

\newcommand{\Gmhat}{\hat{\mathbb{G}}_m}
\DeclareMathOperator{\GL}{GL}

\DeclareMathOperator{\GSp}{GSp}
\DeclareMathOperator{\Sp}{Sp}

\DeclareMathOperator{\Gl}{GL}

\DeclareMathOperator{\tor}{tor}
\DeclareMathOperator{\mult}{mult}
\DeclareMathOperator{\et}{et}
\DeclareMathOperator{\Sym}{Sym}

\DeclareMathOperator{\Gal}{Gal}

\DeclareMathOperator{\Hom}{Hom}

\DeclareMathOperator{\Res}{Res}
\DeclareMathOperator{\Spec}{Spec}
\DeclareMathOperator{\Spf}{Spf}

\DeclareMathOperator{\ord}{ord}
\DeclareMathOperator{\Id}{Id}

\newcommand{\ds}{\displaystyle}

\usepackage[usenames,dvipsnames]{color}  

\begin{document}
\title{Abelian varieties not isogenous to Jacobians over global fields}
\author{Ananth N. Shankar and Jacob Tsimerman}

\date{}

\maketitle

\begin{abstract}
We prove the existence of abelian varieties not isogenous to Jacobians over characterstic $p$ function fields. Our methods involve studying the action of degree $p$ Hecke operators on hypersymmetric points, as well as their effect on the formal neighborhoods using Serre-Tate co-ordinates. Moreover, we use our methods to provide another proof over number fields, as well as proving a version of this result over finite fields.
\end{abstract}

\section{Introduction}

This paper concerns the following question: \emph{Given an algebraically closed field $F$ and $g\geq 4$, does there exist an abelian variety over $F$ of dimension $g$ which is not isogenous to the Jacobian of a stable curve?} The history of the question, as far as we know, is that it was first asked by Nicholas Katz and Frans Oort in the setting of $F=\Qbar$, and then posed more generally by Ching-Li Chai and Oort in \cite{chaioortjacobian}. We mention that an observation of Poonen (unpublished, but credited by Chai and Oort), which helped clarify the situation, was to formulate the question more generally for an arbitrary subvariety of $\Ag$ in place of the Torelli locus. Of course, one can also generalize from $\Ag$ to an arbitrary Shimura variety.

In the case of $F=\Qbar$ the question has been answered in the affirmative: Chai-Oort \cite{chaioortjacobian} proved it conditional on the Andr\'e-Oort conjecture for $\cA_g$, and then the second-named author made their proof unconditional \cite{jacobabvarjacob}. Since then the Andr\'e-Oort conjecture itself has been proved \cite{jacobandreoort} in this case, removing the need for \cite{jacobabvarjacob}. Recently, Masser-Zannier \cite{masserzannier} proved a stronger result, producing abelian varieties not isogenous to Jacobians which are defined over low-degree fields, and in fact showing that `most' such abelian varieties satisfy the condition, when ordered by height in a precise sense.

We now turn our attention to the case of positive characteristic, about which much less is known. Based on arithmetic statistics, we have prior work \cite{ananthjacob} which suggests that the statement should not have affirmative answer for $g\geq 4$.\footnote{For $g\leq 9$, we expect every ordinary abelian variety over $\Fpbar$ to be isogenous to a Jacobian.} In fact, \cite{ananthjacob} suggests a more general result. We first make the following definition: 
\begin{defn}\label{defgenericallyordinary}
Let $X\subset \Ag$ denote a subvariety over a field $k$ of characteristic $p>0$. We say that $X$ is generically ordinary if every irreducible component of $X$ has non-empty intersection with the ordinary locus of $\Ag\times k$. 
    
\end{defn}

Let $D\subset \Ag$ be a generically ordinary divisor defined over $\Fpbar$ (with $g>1$). Then, we expect that every $x\in \Ag(\Fpbar)$ should be isogenous to some $y\in D(\Fpbar)$. Bjorn Poonen has similar heuristics in an unpublished pre-print, which agree with our expectations. 

In the way of unconditional results, there is only a result of Chai-Oort \cite[\S4]{chaioortjacobian} which deals with the analogous situation of a curve inside $X(1)^2$ (which produces non-ordinary points) and a result of the authors \cite[Thm 4.1]{ananthjacob} for a specific hypersurface in $X(1)^N$ for $N\geq 270$, using additive combinatorics. Relatedly, work of the first author, Asvin G. and Q. He \cite{hilbertsurfaces} shows that if $C_1,C_2$ are two generically ordinary curves contained in a Hilbert modular surface associated to a real quadratic field split at $p$, then there are infinitely many closed points on $C_1$ isogenous to some point on $C_2$. Asvin G. in \cite{Asvin} proves results similar to (and stronger than) those in \cite{hilbertsurfaces} pertaining to just-likely and unlikely intersections on products of modular curves.

The setting most analogous to number fields is function fields in one variable, and thus the natural analogue to $\Qbar$ is $\ol{\F_p(T)}$. The main goal of this paper is to deal with precisely that case.  
Our main theorem is as follows:


\begin{theorem}\label{thm:main1}
Let $F=\ol{\F_p(t)}$, and $g>1$. Let $D\subset\Ag/F$ be a divisor. There exists an $F$-valued point of $\Ag$ which is not isogenous to any $F$-point of $D$. 
\end{theorem}

As a stepping stone to our main theorem, we first prove the following (which is sufficient for the case of the Torelli locus):

\begin{theorem}\label{thm:main}
Let $F=\ol{\F_p(t)}$, and $g>1$. Let $D\subset\Ag/\F_q$ be a divisor, where $q$ is a power of $p$. There exists an $F$-valued point of $\Ag$ which is not isogenous to any $F$-point of $D$.  
\end{theorem}
The heuristics offered in \cite{ananthjacob} suggest that Theorem \ref{thm:main} is the strongest possible theorem that is true in positive characteristic; namely, that $\ol{\F_p(t)}$ is the smallest algebraically closed field over which one can expect the existence of an abelian variety not isogenous to any point on $D$. 

Of course, one may instead formulate the above theorem in terms of curves contained in $\Ag$ by using the familiar translation between curves and function fields.

\subsection{Other results}
Theorem \ref{thm:main} yields yet another proof of the number field version of Theorem \ref{thm:main} (namely, the existence of abelian varieties over number fields not isogenous to Jacobians):

\begin{theorem}\label{new number field}
Let $D\subset \cA_g$ denote a divisor defined over $\overline{\mathbb{Q}}$. Then there exists a Hodge-generic point $x\in \cA_g(\overline{\mathbb{Q}})$ not isogenous to any point $y\in D$. 
\end{theorem}

This method also yields the following result over finite fields: 

\begin{theorem}\label{finitefields}
Let $D\subset \Ag$ denote a divisor over $\Fpbar$. Then, there exists an ordinary $x\in \Ag(\Fpbar)$ such that $\tau(x) \not\subset D$ for every prime-to-$p$ Hecke operator $\tau$. 
\end{theorem}

Theorem \ref{finitefields} acts as a substitute over $\Fpbar$, because as stated before we expect that every ordinary $x\in \Ag(\Fpbar)$ is isogenous to some point in $D(\Fpbar)$.  

We remark that Theorem \ref{finitefields} also highlights one of the differences/difficulties in handling the $\Fpbar$ case: The Galois orbits are just much smaller, so that $\tau(x)$ cannot be a single Galois orbit, unlike the case of number fields or function fields.

\subsection{Ideas of Proofs}
\subsubsection{Sketch of Proof of Theorem \ref{thm:main}}

Fix a divisor $D\subset(\Ag)_{\F_q}$. Our first observation is that in the function field case one can impose very strong `local' conditions\footnote{It is unclear to us whether one can impose analogously strong local conditions  over number fields.} on our $F$-point, which prohibits such a point from being contained in $D$. Thinking of $F$-valued points as curves $C\subset\Ag$, this amounts to constructing a curve $C$ which is highly singular at a given point $Q$, such that the formal neighborhood of $C$ at $Q$ contains much of the formal neighborhood of $\Ag$ at $Q$. There are some difficulties with implementing this strategy, namely: 
\begin{enumerate}

    \item While the local structure of prime-to-$p$ Hecke operators is very well understood, $p$-power Hecke operators behave poorly in characteristic $p$. These are not \'etale (or even finite) and strongly distort the local structure. 
    
    \item It is a priori unclear how to insist that every \emph{irreducible} component of a Hecke operator applied $C$ also has this property. A priori, different branches of $C$ through a point $x$ could separate upon applying a Hecke operator.

\end{enumerate}

We overcome the second difficulty by insisting that the curve has maximal monodromy by using Lefschetz theorems in conjunction with Bertini-style theorems due to Poonen and Charles-Poonen. Indeed, if $C$ has maximal monodromy, then $\tau(C)$ is itself irreducible (this is well known, but we include a proof in Section 5 for completeness - see Proposition \ref{prop:heckeirreducible}). The first difficulty is therefore the main one, and our idea in overcoming it is to use \emph{hypersymmetric} abelian varieties (see Definition \ref{hypersymmetricabvars} for the definition). The $p$-power Hecke orbit of a hypersymmetric point consists of points defined over the same field as the original point (similar in spirit to the prime-to-$p$ Hecke orbit of a supersingular point). 
In particular, the orbit of a hypersymmetric point under $p$-power isogenies is finite.

This prevents `escape to infinity' and allows us to impose certain local conditions on this entire finite set all at once, which we can then control. To impose our local conditions, we use the bilinear structure on the tangent space (and in fact the entire formal neighborhood) inherited from the Serre-Tate co-ordinates to isolate certain favored directions which behave stably under $p$-power isogenies.

We remark that in an unpublished preprint, Poonen is able to establish the existence of an abelian variety over a global function field not isogenous by a \emph{prime-to-}$p$ isogeny to any Jacobian. His method relies on a careful analysis of the $p$-adic geometry of prime-to-$p$ Hecke operators on $\cA_g$. 
\subsubsection{Sketch of Proof of Theorem \ref{thm:main1}}

It turns out that the hardest case of this theorem is Theorem \ref{thm:main}. Our approach is essentially to notice that $\Fpbar(T)$ contains many distinct subfields - in fact, copies of itself: $\Fpbar(P(T))$ for polynomials $P$. One can show that if $D$ is defined over some subfield $E\subset F$, then we may find a smaller subfield $E'\subset E$ such that the intersection of $D$ with its own Galois conjugates over $E'$ is defined over $\Fpbar$. But if $C$ is a curve defined over $E'$, then $C\subset D$ implies that $C$ is contained in the intersection of $D$ with all its Galois conjugates over $E'$. One must be a bit careful, and we work with curves instead of function fields, but this idea can essentially be pushed through to reduce the proof of Theorem \ref{thm:main1} to Theorem \ref{thm:main}. 

\subsubsection{Theorem \ref{new number field}}

This same idea can be carried out (with a lot less difficulty, as we no longer face any characteristic $p$ issues) to find a curve $C\subset \Ag$ none of whose Hecke translates are contained in $D$, where $C$ is now a curve defined over a \emph{Number Field}. We use the existence of such a curve, along with a strong "big monodromy" result due to Zywina, to provide an intersection-theoretic proof of the existence of Number Field-valued points of $\Ag$ that are not isogenous to any point of $D$. This method also goes through to prove Theorem \ref{finitefields}.

We remark that Masser-Zannier in \cite{masserzannier} have established Theorem \ref{new number field} by using Hodge-generic points, and in addition, show that ``most'' points of $\cA_g$ satisfy this property. We remark that they need less control on the monodromy as compared to our proof, which they use in conjunction with the Pila-Zannier method to establish their result. Specifically, they only need the monodromy to be large enough such that under an isogeny of degree $n$ the degree of the field of definition grows by a power of $n$. Once they have this, they use the Pila-Wilkie method to create algebraic curves in the Siegel domain and then use transcendence theory.

By contrast, our method requires for the monodromy to be maximal with almost no `wiggle room', but the proof method does not go through functional transcendence theory and instead uses intersection theory through a Bezout-argument.

We further remark that both methods can be used to effectively bound the degree of the field of definition of the resulting points, and to obtain infinitely many points of such degree which are not isogenous to any point in the given divisor.

\subsection{Structure of Paper}
Section 2 is devoted to preliminaries.
In section 3 we show how to construct abelian varieties over curves
with prescribed local and monodromy conditions. In section 4 we review the theory of Serre-Tate coordinates, introduce the notion of primitive Serre-Tate directions and study how they behave under isogenies. This is the technical heart of the paper. We then prove Theorem 1.2 and Theorem 1.1 in Section 5. Finally, in Section 6 we give a couple of other results: we explain how our methods can give a rather short proof of the same result over number fields (although relying on a strong monodromy theorem of Zywina) and we prove Theorem \ref{finitefields}.

\subsection{Acknowledgements}
The first author is grateful to Bjorn Poonen and Umberto Zannier for useful conversations. The second author is grateful to Frans Oort for introducing him to the problem and for many useful discussions. We are grateful to Daniel Litt for providing us with a reference for tame Lefschetz theorems. We are also grateful to the referees for extremely careful readings and their comments which have truly helped improve this paper. A.S. was partially supported by NSF grant DMS-2100436, NSF career grant DMS-2338942, and a Sloan research fellowship. 

\section{Preliminaries on moduli Spaces of abelian varieties} 

\subsection{Moduli spaces of abelian varieties}

Throughout the rest of the paper we consider a fixed prime $p$, and we choose an auxiliary odd prime $r\neq p$. 





We let $\Agr$ denote the (fine) moduli space of principally polarized abelian schemes with full symplectic level $r$ structure over $\Z[\zeta_r, \frac1r ]$ (see \cite[Chapter IV, Definition 6.1 (i)]{chaifaltings} for the definition of full symplectic level $r$ structure). We note that the geometric fibers of $\Agr$ over $\Spec \Z[\zeta_r, \frac1r]$ are irreducible (see \cite[Chapter IV Corollary 6.8]{chaifaltings}). For any $\Z[\zeta_r, \frac1r]$-scheme $S$ and a principally polarized abelian scheme $A/S$ having a full symplectic level $r$ structure, we let $[A]$ denote the $S$-valued point of $\Agr$ that induces the principally polarized abelian scheme $A$ along with its level structure\footnote{In particular, the notation suppresses the level structure for brevity.}. 

There exists a smooth toroidal compactification of $\Agr$ over $\mathbb{Z}[\zeta_r,\frac1r]$ that we denote by $\Agrtor$, with the boundary $Z$ being a normal-crossings divisor (see \cite[Chapter IV, Theorem 6.7]{chaifaltings}, or \cite[Theorem 2]{keerthitoroidalcompact}). Now for each integer $n$ we have an $n$-torsion local system $\cL_n$ over $\Agr/\Z[\zeta_r,\frac1{nr}]$. The assumptions above yield the following result: 

\begin{theorem}[Grothendieck, Raynaud]\label{thm:tamelyramified}
Suppose that $n$ is an integer such that $ p\nmid n$, and let $k$ be a field having characteristic $p$. Then the local system $\cL_n/\Agr_k$ is tamely ramifed at the boundary $Z$. 
\end{theorem}
\begin{proof}
This result follows almost directly from results of Grothendieck and Raynaud. Let $K$ denote the function field of $\Agr_k$, let $v$ be a discrete valuation of $K$ arising from an irreducible component of $Z$, and suppose that $K_v$ is the $v$-adic completion of $K$. Let $I \subset \Gal(\overline{K_v}/K_v)$ denote the inertial subgroup. It suffices to prove that the image of $I$ in $\GL_{2g}(\Z/n\Z)$ doesn't contain any element having order $p$. Indeed, this is trivially true for $n = r$. We then have that the universal abelian variety over $K_v$ has semistable reduction at $v$ (see \cite[Theorem 3.5]{Alice}, and \cite[Proposition 4.7]{Grothendieck}). Grothendieck's semistable reduction theorem then implies that for any $n$ not a multiple of $p$, the action of $I$ on $A[n]$ is unipotent, and therefore the image of $I$ in $ \GL_{2g}(\Z/n\Z)$ is prime-to-$p$. The result follows. 
\end{proof}

Continuing to work over a field $k$ of characteristic $p$, we let $\Agr_k^{\ord}$ denote the ordinary locus. Define $\cL_{p^m}$ over $\Agr_k^{\ord}$ to be the dual of the connected component of the $p^m$ torsion of the universal abelian scheme restricted to $\Agr_k^{\ord}$. It is a $p^m$-torsion \'etale local system on $\Agr_k^{\ord}$. 

\subsection{Hecke Operators}
We will continue to work over a field $k$ of characteristic $p$.
\subsubsection{Prime-to-$p$ Hecke operators}

Consider a prime number $\ell \nmid  pr$, and let $\gamma \in \GSp_{2g}(\Q_{\ell}) \cap M_{2g}(\Z_{\ell})$. Consider the double coset $\GSp_{2g}(\Z_{\ell})\gamma \GSp_{2g}(\Z_{\ell})$, and let $\gamma_1,\gamma_2, \hdots \gamma_m \in \GSp_{2g}(\Q_{\ell}) \cap M_{2g}(\Z_{\ell})$ denote a set of right-coset representatives, i.e. $\GSp_{2g}(\Z_\ell)\gamma \GSp_{2g}(\Z_\ell)= \bigsqcup_{i} \gamma_i \GSp_{2g}(\Z_\ell)$. Consider the kernel of $\gamma_i \mod \ell^{m}$. There is a canonical inclusion\footnote{This is induced by the canonical inclusion $\Z/\ell^m \rightarrow \Z/\ell^{m+1}$} $\ker \gamma_i \mod \ell^{m} \rightarrow \ker \gamma_i \mod \ell^{m+1}$. There exists a large enough integer $n_i$ such that this inclusion is an isomorphism for every $m\geq n_i$, i.e. this kernel eventually stabilizes. Pick an integer $n$ that works for all $i$. 

For any $k$-scheme $S$, and an $S$-point of $\Agr$, we obtain an abelian scheme $A/S$. Identify the $\ell$-adic Tate module of $A$ with $\Z_{\ell}^{2g}$ as symplectic spaces. Then, each $\gamma_i$ defines an \'etale subgroup $G_i\subset A[\ell^n]$, $A_i = A/G_i$ is a principally polarized abelian scheme defined over an \'etale cover $S_i$ of $S$, and the quotient map $A\rightarrow A_i$ respects the principal polarizations upto scaling (see \ref{subsubsec: polarized isogenies} below). Further, the level $r$-structure on $A$ uniquely determines a level-$r$ structure on each $A_i$. We define the Hecke operator $\tau_{\gamma}$ to be the correspondence from $\Agr$ to itself, defined by the formula $\tau_{\gamma}[A] = \bigsqcup_{i} [A_i]$. 

We may similarly define $r$-power Hecke operators by going through the same process, except we replace $\GSp_{2g}(\Z_{r})$ with the principal congruence subgroup $\Gamma'_g(\Z_r)\subset \GSp_{2g}(\Z_r)$ consisting of those elements of $\GSp_{2g}(\Z_r) $ that are congruent to 1 mod $r$.

\subsubsection{$p$-power Hecke operators}
The characteristic $p$ theory of $p$-power Hecke operators is complicated. We will follow \cite[Chapter VII Section 4]{chaifaltings}, which defines $p$-power Hecke operators on the ordinary locus $\Agr_k^{\ord}$ of $\Agr_{k}$. We describe these Hecke operators in more detail in \S 4.2, after recalling background on ordinary $p$-divisible groups and Serre-Tate coordinates. 

\subsubsection{Polarized isogenies}\label{subsubsec: polarized isogenies}
Let $A/S$ and $B/S$ denote two principally polarized ordinary abelian schemes with full level-$r$ structure, where $S$ is a $k$-scheme. Let $\lambda_A,\lambda_B$ denote the principal polarizations on $A$ and $B$. A polarized isogeny $\phi: A\rightarrow B$ is an isogeny $\phi$ such that $\phi^*(\lambda_B) = n \lambda_A$, where $n$ is some positive integer. We note that $A$ is isogenous to $B$ via a polarized isogeny respecting the level-$r$ structure precisely when $[B]\in \tau([A])$, where $\tau$ is a product of $\ell$-power and $p$-power Hecke operators, for finitely many primes $\ell\neq p$. Specialize to the case when $S = \Spec \ol{\F_p(t)}$ and suppose further that $A,B$ have endomorphism ring $\Z$. Then, it is easy to see that any isogeny from $A$ to $B$ must necessarily be a polarized isogeny. 

\subsubsection{Monodromy of these local systems}
We remark that all these local systems have ``maximal monodromy''. Namely, that the natural map $\pi_1(\Agr_{\ol{k}}^{\ord},x) \rightarrow \Gamma_g(\Z_r) \times \prod_{\ell \neq p,r} \Sp_{2g}(\Z_{\ell}) \times \GL_g(\Z_p)$ induced by these local systems is surjective, where $\Gamma_g(\Z_r) = \Gamma'_g(\Z_r)\cap \Sp_{2g}(\Z_r)$ (see \cite{chaifaltings}). 

\section{Constructing families of abelian varieties with specified local conditions}

Throughout this section we work purely in characteristic $p$, so that when we write $\Agr$ we mean $\Agr_{\F_p[\zeta_r]}$.

In what follows, when referring to the \emph{geometric fundamental group} of a geometrically irreducible curve $C$ over a field $k$, we mean $\pi_1(\eta)$ where $\eta$ is the generic point of $C_{\ol k}.$ Throughout, our curve will map to $\Agr$. The local systems on $\Agr$ naturally induce homomorphisms from the geometric fundamental group of $C$ to $\Sp_{2g}(\Z_{\ell})$ and $\GL_g(\Z_p)$ (and their quotients).


We shall need the following result\footnote{We work with the adjoint groups at almost all $\ell$ because each $\Sp_{2g}(\Z_\ell)$ has a $\Z/2\Z$ factor, whereas we want to have distinct Jordan-Holder factors.}:

\begin{theorem}\label{congmain}
Fix $g>1$. Let $x_1,\dots,x_n$ be elements of $\Agr^{\ord}(\Fpbar[[t]])$ whose differentials are injective, so that their images in $\Agr^{\ord}\bigg(\Fpbar[t]/(t^2)\bigg)$ are not contained in $\Agr^{\ord}(\Fpbar)$. Fix also positive integers $m_1,\dots,m_n$. There exists a curve $C\subset\Agr^{\ord}_{\Fpbar}$ such that
\begin{enumerate}
    \item  $C$ admits $\Spec\Fpbar[[t]]$ points specializing to all of the $x_i \bmod t^{m_i}$, and 
    \item The geometric fundamental group of $C$ surjects onto $\Gamma_g(\Z_r)\times\prod_{\ell\neq p,r}\bbP(\Sp_{2g}(\Z_\ell))\times\GL_g(\Z_p)$
\end{enumerate}

where $$\bbP(\Sp_{2g}(\Z_\ell)):=\Sp_{2g}(\Z_\ell)/\{ \pm \Id\}$$
\end{theorem}

In the rest of this section we prove Theorem \ref{congmain}. The idea of the proof is twofold: We construct $C$ by repeated intersections with smooth hypersurfaces. This allows us to meet condition (1) by keeping track of local conditions. As for condition (2), we show that it is generic: For the prime-to-$p$ part we use tame Lefschetz theorems (see \cite[Theorem 7.3]{Esnault} and \cite{EKin}), and we prove the geometric fundamental group of $C$ surjects onto the $p$-part by having $C$ contain suitable $\F_q$ points whose Frobenius-images generate the $p$-part.

Let $\F_q$ be a finite field over which $\bigcup_i (x_i\bmod{t^{m_i}})$ is defined. By increasing $q$ if necessary, we pick finitely many points $Q_i \in \cA_g[r](\F_q)$ whose Frobenius-images contain every conjugacy class in $\Gl_g(\Z/p^2\Z)$, and we set $R = \cup_i Q_i$. 

\begin{lemma}\label{lem: fullgeommon}
Let $E\subset\Agr^{\ord}$ be an irreducible curve defined over $\F_q$ whose smooth locus contains $R$. Then the geometric fundamental group of $E$ surjects onto $\Gl_g(\Z/p^2\Z)$.
\end{lemma}
\begin{proof}
For ease of notation, we let $f: \pi_1(E,\ol{Q}) \rightarrow \Gl_g(\Z/p^2\Z)$ denote our map. By removing points we may assume $E$ is smooth. Pick a point $Q\in R$ such that the Frobenius at $Q$ maps trivially to $\Gl_g(\Z/p^2\Z)$ under $f$. We have the sequence $$1\ra\pi_1(E_{\ol{\F}_p},\ol{Q})\ra\pi_1(E,\ol{Q})\ra \pi_1(\F_q)\ra 1$$ where the map $\pi_1(E,\ol{Q})\ra \pi_1(\F_q)$ admits a splitting $i_Q$ via the point $Q$.

Now since $R\subset E$ it follows that $f(\pi_1(E,\ol{Q}))$ intersects every conjugacy class of $\Gl_g(\Z/p^2\Z)$, and therefore $f$ is surjective. It then follows that the image of $\pi_1(E_{\Fpbar},\ol Q)$ is a normal subgroup $H$ of $\Gl_g(\Z/p^2\Z)$, and that $H$ together with $f(i_Q(\pi_1(\F_q)))$ generates $\Gl_g(\Z/p^2\Z)$. However, by the choice of $Q$, $f(i_Q(\pi_1(\F_q)))$ is trivial, and so $H$ must be all of $\Gl_g(\Z/p^2\Z)$.
Since $E$ is smooth (and in particular normal) the geometric fundamental group surjects onto $\pi_1(E_{\Fpbar}, \bar{Q})$, and therefore onto $H$ which 
completes the proof.


\end{proof}


\begin{lemma}\label{lem: samequot}
There are finite quotient groups $H$ and $H'$ of $\Gamma_g(\Z_r)\times\prod_{\ell\neq p,r}\bbP(\Sp_{2g}(\Z_\ell))$ and $\GL_g(\Z_p)$ respectively, such that if the same group $N$ occurs as a quotient of both $\Gamma_g(\Z_r)\times\prod_{\ell\neq p,r}\bbP(\Sp_{2g}(\Z_\ell))$ and $\GL_g(\Z_p)$, then the quotient maps must factor through $H$ and $H'$. 
\end{lemma}

\begin{proof}
Note that all but finitely many $\bbP(\Sp_{2g}(\Z/\ell\Z))$ are simple Chevalley groups. Moreover, the only Jordan--Holder factors of $\bbP(\Sp_{2g}(\Z_\ell))$ are
$\bbP(\Sp_{2g}(\Z/\ell\Z))$ and $\Z/\ell\Z$. 

Therefore any quotient $N$ of $\Gamma_g(\Z_r)\times\prod_{\ell\neq p,r}\bbP(\Sp_{2g}(\Z_\ell))$ which does not annihilate $\bbP(\Sp_{2g}(\Z_\ell))$ must admit either $\bbP(\Sp_{2g}(\Z/\ell \Z))$ or $\Z/\ell\Z$ as a quotient, and thus $\GL_g(\Z_p)$ must also admit one of them as a quotient. However, there are only finitely many of these groups that $\GL_g(\Z_p)$ admits as a quotient, and so all but finitely many\footnote{We emphasize the point that the finite set of primes $\ell$ does not depend on $N$.} of the $\bbP(\Sp_{2g}(\Z_\ell))$ must be in the kernel of a surjective map to any such $N$.

Now for any $\ell\neq p,$ and any $\alpha > 1$, note that $\bbP(\Sp_{2g}(\Z_\ell))$ has a normal-cofinite pro-$\ell$ subgroup $X_{\ell,\alpha}$, all of whose elements are $\ell^\alpha$-powers in $\bbP(\Sp_{2g}(\Z_\ell))$. We choose $\alpha$ so that $\ell^{\alpha}> |\GL_g(\F_p)|$. If the kernel of the quotient map to $N$ doesn't contain $X_{\ell,\alpha}$, it follows that $N$ must contain an element whose order is a multiple of $\ell^{\alpha}$, which is impossible as $N$ is also a quotient of $\GL_g(\Z_p)$. The analogous fact also holds for $\Gamma_g(\Z_r)$, which yields the existence of $H$. The existence of $H'$ follows in the same way by noticing that $\Gl_g(\Z_p)$ admits a pro-$p$ normal subgroup $X_{p,a}$, all of whose elements are $p^a$-powers in $\GL_g(\Z_p)$. We now choose $a$ such that $p^a > |H|$ and we may take $H'$ to be the quotient $\GL_g(\Z_p)/X_{p,a} $.

\end{proof}


By possibly increasing $\F_q$ we increase $R$ to also contain (finitely many extra) points in $\Agr^{\ord}(\F_q)$ whose Frobenius-images contain every conjugacy class in $H\times H'$. By the same proof as Lemma \ref{lem: fullgeommon}, every curve defined over $\F_q$ whose smooth locus contains $R$ as a divisor has geometric fundamental group that surjects onto $H\times H'$. We moreover insist that $R$ is distinct from the closed points in the images of the $x_i$. We are now ready to complete the proof of the main result of this section. 

\begin{proof}[Proof of Theorem \ref{congmain}]

First, consider a series of blowups at smooth points of $\Agrtor$ to obtain a smooth variety $B$ in which all the $x_i$ separate. Next we pick an ample class $L$ of $B$. We shall construct curves in $B$ defined over $k$ by repeatedly intersecting divisors of sections of $L^m$ for large $m$, in a way that is generic in the sense of \cite{poonen}. By \cite[Theorem 1.3]{poonen} and \cite[Theorem 1.1]{charlespoonen}, a density 1 of such curves will be smooth and intersect the boundary divisor smoothly.  Thus by Esnault-Kindler \cite[Theorem 1.1a]{EKin}
together with Theorem \ref{thm:tamelyramified}, the geometric fundamental group of these curves will surject onto $\Gamma_g(\Z_r)\times\prod_{\ell\neq p,r}\bbP(\Sp_{2g}(\Z_\ell)).$ Pick any such curve which also contains $R$ - which is one of a positive density of such curves. By Lemma \ref{lem: fullgeommon}, we have that the geometric fundamental group surjects onto $\GL_g(\Z/p^2)$. The surjectivity onto $\GL_g(\Z/p^2)$ implies surjectivity onto $\GL_g(\Z_p)$. We therefore have that the geometric fundamental group surjects onto $\GL_{g}(\Z_p)$ and $H\times H'$. 

Goursat's Lemma now states that there are isomorphic profinite quotients $N,N'$ of $\Gamma_g(\Z_r)\times\prod_{\ell\neq p,r}\bbP(\Sp_{2g}(\Z_\ell))$ and $\Gl_g(\Z_p)$ such that the image of geometric fundamental group is the pullback of the graph of an isomorphism between them. Let $f: N\rightarrow N'$ denote this isomorphism. By  Lemma \ref{lem: samequot}, the quotient maps factor through $H\xrightarrow{\pi} N$ and $H'\xrightarrow{\pi'} N'$, and so the image of the geometric fundamental group in $H\times H'$ consists of elements $\{(h,h')\}$ satsifying $f(\pi(h)) = \pi'(h')$. However, we have already arranged for the geometric fundamental group to surject onto $H\times H'$. Therefore, we have that $N$ and $N'$ are trivial. Therefore, the geometric fundamental group surjects onto $\Gamma_g(\Z_r)\times\prod_{\ell\neq p,r}\bbP(\Sp_{2g}(\Z_\ell))\times \Gl_{g}(\Z_p)$ as desired.

\end{proof}
\medskip


\section{Hecke operators in characteristic $p$}

Throughout this section, we work purely in characteristic $p$ and will again write $\Agr$ when we mean $\Agr_{\F_p[\zeta_r]}$. We will let $\Agr^{\ord}$ denote the ordinary locus of $\Agr$. 

\subsection{Background on Serre-Tate coordinates}
We briefly recall Serre-Tate coordinates and describe the local action of $p$-power Hecke operators in terms of these coordinates. For a thorough treatment of Serre-Tate coordinates, see \cite{katzserretate} or \cite{chaioort}.

Let $\pdiv$ denote an ordinary $p$-divisible group over an algebraically closed field $k$ with dimension $g$ and height $2g$. We have that $\pdiv \simeq \pdiv^{\mult} \times \pdiv^{\et}$, where $\pdiv^{\mult} \simeq \mu_{p^{\infty}}^g$ and $\pdiv^{\et} \simeq (\Q_p/\Z_p)^g$. The $p$-divisible groups $\pdiv^{\mult}$ and $\pdiv^{\et}$ are rigid, and the functoriality of the connected-\'etale 
exact sequence implies that deformations of $\pdiv$ to an Artin local $k$-algebra $R$ are in bijection with extensions of $\pdiv^{\et}\times \Spec R$ by $\pdiv^{\mult}\times \Spec R$. In particular, the deformation space of $\pdiv$ has a natural group structure. 

Let $T_p(\pdiv)$ denote the Tate-module of $\pdiv^{\et}$, and let $T_p(\pdiv^{\vee})$ denote the Tate-module of $(\pdiv^{\mult})^\vee$, where $^\vee$ denotes taking the Cartier dual. Then, the formal deformation space of $\pdiv$, which we have already seen has the structure of a group, actually has the structure of a formal torus, and is canonically isomorphic to $\Hom(T_p(\pdiv)\otimes T_p(\pdiv^{\vee}),\Gmhat)$. It will later also be convenient to use the canonical identification of this torus with $(T_p(\pdiv)^* \otimes T_p(\pdiv^\vee)^*) \otimes \Gmhat$. Here, $^*$ denotes taking the linear dual of a $\Z_p$-module.

Specifically, let $R$ denote any Artin local $k$-algebra, with maximal ideal $\fm$. Then, the deformations of $\pdiv$ to $R$ are in bijection with bilinear maps $q: T_p(\pdiv) \times T_p(\pdiv^{\vee}) \rightarrow 1+\fm$. Therefore, a $\Z_p$-basis $e_i = \{e_1\hdots e_g\}$ of $T_p(\pdiv)$ and $m_j = \{m_1\hdots m_g\}$ of $T_p(\pdiv^{\vee})$ yield coordinates $Q=\{q_{ij}\}$ on the deformation space, i.e. the characteristic $p$ deformation space is isomorphic to $\Spf k[[q_{ij}-1]] $, and an $R$-valued point of the deformation space corresponds to a choice of $g^2$ elements in $1 + \fm$, which is the same data as a bilinear map $q: T_p(\pdiv) \times T_p(\pdiv^{\vee}) \rightarrow 1+\fm$. The same analysis holds when $R$ is a complete local $k$-algebra. We will think of $Q$ as an element of $M_g((1+\fm))$. We note that the natural $\Z_p$-module structure on $1+\fm$ canonically equips $M_g((1+\fm))$ with the structure of an $M_g(\Z_p)$ bi-module. We now specialize to the case of $R = k[[t]]$. Given $Q =\{q_{ij}(t) \} \in M_g((1+\fm))$, we define $\frac{1}{p^n}\cdot Q$ as $\{ q_{ij}(t^{\frac{1}{p^n}})\} \in M_g((1+ \fm_{n}))$, where $\fm_n \subset k[[t^{\frac{1}{p^n}}]]$ is the maximal ideal. This allows us to ``act'' on the deformation space by $\GL_g(\Q_p)$ (on both sides), except that the output will valued in $k[[t^{\frac{1}{p^{\infty}}}]]$.

Let $\pdiv'$ denote another $p$-divisible group and let $\phi: \pdiv\rightarrow \pdiv'$ be an isogeny. We pick bases $\{e'_i,m'_j\}$ for $T_p(\pdiv'),T_p(\pdiv^{\prime \vee})$, and let $X$ and $Y$ denote the matrices (necessarily with non-zero determinant) of $\phi: T_p(\pdiv)\rightarrow T_p(\pdiv')$ and $\phi^{\vee}: T_p(\pdiv^{\prime \vee})\rightarrow T_p(\pdiv^\vee)$ in these coordinates. Let $\lift/k[[t]]$ deform $\pdiv$, and let $Q = \{q_{ij}(t)\}$ denote its Serre-Tate coordinates. The following result is well known (see for example \cite[Theorem 2.1]{katzserretate} or \cite[Proposition 2.23]{chaioort}).
\begin{proposition}\label{prop:coordinatesafterisog}
Notation as above, and consider the matrix $Q' = {X^T}^{-1}QY$. Note that a priori the entries of $Q'$ are valued only in $k((t^{\frac{1}{p^{\infty}}}))$.
\begin{enumerate}
    \item Suppose that the entries of $Q'$ are valued in $1+tk[[t]]$. Then, $\phi$ lifts uniquely to an isogeny $\lift \rightarrow \lift'$ over $k[[t]]$ (where $\lift'$ is necessarily a deformation of $\pdiv'$) and the Serre-Tate coordinates (with respect to $\{e'_i,m'_j \}$) of $\lift'$ are $Q'$.
    
    \item If the entries of $Q'$ aren't valued in $1 + tk[[t]]$ (eg. if some entry equals $(1+t)^{\frac{1}{p}}$), then $\phi$ doesn't lift to an isogeny over $k[[t]]$ with source $\lift$. However,  let $n=n(\lift,\phi)$ denote the smallest positive integer such that the entries\footnote{For example, if $Q'$ were the $1\times 1$ matrix with entry $(1+t)^{\frac{1}{p}}$, then $n=1$ and $Q'(t^p)$ has entry $(1+t^p)^{\frac{1}{p}} = (1+t)$.} of $Q'(t^{p^n})$ are valued in $1+tk[[t]]$. Then, $\phi$ lifts uniquely to an isogeny $\lift \times_{\Spf k[[t]]} \Spf k[[s]] \rightarrow \lift'$, where the map $\Spf k[[s]] \rightarrow \Spf k[[t]]$ is given by $t\mapsto s^{p^n}$, and $\lift'$ is a deformation of $\pdiv'$ to $k[[s]])$. The Serre-Tate coordinates of $\lift \times_{\Spf k[[t]]} \Spf k[[s]]$ are $Q(t) = Q(s^{p^n})$, and of $\lift'$ are $Q'(s) = {X^T}^{-1}Q(s^{p^n})Y$. 
\end{enumerate}

 In either case, we will let $\liftphi$ denote the isogeny that lifts $\phi$, which is already defined over $\Spf k[[t]]$ in the first case, but which is only defined over $\Spf k[[s]] = \Spf k[[t^{\frac{1}{p^n}}]] $ in the second case. 
\end{proposition}

Suppose now that $\pdiv$ is equipped with a principal polarization, $\lambda$. This yields canonical isomorphisms $T_p(\pdiv)\rightarrow T_p(\pdiv^{\vee})$ and $T_p(\pdiv)^*\rightarrow T_p(\pdiv^\vee)^*$, both of which we will also denote by $\lambda$. Suppose that we have chosen coordinates such that $m_i = \lambda(e_i)$. Let $\{\epsilon_i\}$ be a basis of $T_p(\pdiv)^*$ dual to $\{e_i\}$, and let $\mu_i = \lambda(\epsilon_i)$ (note that $\{\mu_i\}$ is a basis of $T_p(\pdiv^{\vee})^*$ dual to $\{m_i\}$). Then, $\lambda$ lifts to a deformation of $\pdiv$ precisely when $q_{ij} = q_{ji}$ (\cite[Corollary 2.2.4]{chaioort}). The deformation space of $(\pdiv,\lambda)$ is $\Hom^{\Sym}(T_p(\pdiv)\otimes T_p(\pdiv^\vee),\Gmhat) \subset \Hom(T_p(\pdiv)\otimes T_p(\pdiv^\vee),\Gmhat)$, the set of all symmetric maps. This space is canonically isomorphic to $\big(T_p(\pdiv)^*\otimes T_p(\pdiv^\vee)^*\big)^{\Sym}  \otimes \Gmhat$, where $(T_p(\pdiv)^*\otimes T_p(\pdiv^\vee)^*)^{\Sym}$ is the $\Z_p$-span of $\{\epsilon_i \otimes \mu_j + \epsilon_j\otimes \mu_i, \epsilon_i\otimes \mu_i \}$. In coordinates, this space equals $\Spf k[[q_{ij}]]/(\{q_{ij} - q_{ji} \})$.

\subsection{Description of $p$-power Hecke operators on the ordinary locus.}
We now briefly describe $p$-power Hecke operators on ${\Agr}^{\ord}$. For further details, the interested reader may consult \cite[Chapter VII Section 4]{chaifaltings}. 

Let $(\pdiv,\lambda),(\pdiv',\lambda')$ denote ordinary principally polarized $p$-divisible groups as above, and let $\phi: \pdiv \rightarrow \pdiv'$ denote an isogeny such that $\phi^*(\lambda') = p^n\lambda$. Suppose that we choose bases for $T_p(\pdiv),T_p(\pdiv^\vee),T_p(\pdiv')$ and $T_p(\pdiv^{\prime,\vee})$ which are compatible with $\lambda,\lambda'$. Let $X$ and $Y$ denote the matrices of the maps $T_p(\pdiv) \rightarrow T_p(\pdiv')$ and $T_p(\pdiv^{\prime,\vee}) \rightarrow T_p(\pdiv^{\vee})$ induced by $\phi$. The condition $\phi^{*}\lambda' = p^n\lambda$ is equivalent to $YX = p^n\Id$. We say that the matrix of $\phi$ is $(X,Y)$. 

To that end, let $(n:n_1,n_2,\hdots n_g)$ denote a sequence of integers such that $n\geq n_1\geq n_2 \hdots \geq n_g\geq 0$. Consider all polarized isogenies $\{\phi: \pdiv \rightarrow \pdiv' \}$ between ordinary objects such that the matrix of $\phi$ is $(X,Y)$ where $YX = p^n\Id$ and 
\[
X \in \GL_g(\Z_p)
\left[
\begin{array}{cccc}
p^{n_1}&&&\\
&p^{n_2}&&\\
&&\ddots&\\
&&&p^{n_g}\\
\end{array}
\right] \GL_g(\Z_p).
\]\label{hecketype}
We say that the \emph{type} of $\phi$ is $(n:n_1,n_2,\hdots n_g)$. Given a principally polarized ordinary abelian variety $A$, consider all polarized $p$-power isogenies $\phi_i: A\rightarrow A_i$ of type $(n:n_1,n_2,\hdots n_g)$. Define the Hecke correspondence $\tau$ of \emph{type} $(n:n_1,n_2, \hdots n_g)$ as $\tau([A]) = \bigsqcup [A_i]$, where the $A_i$ are as above.

Note that this exhausts all polarized $p$-power isogenies as we range over all types, since the union of all such double cosets is $M_g(\Z_p)\cap \GL_g(\Q_p)$.

We note that the Hecke operator $\tau$ is Frobenius precisely when it has type $(1:0,0,\hdots 0)$ and is Verschiebung precisely when it has type $(1:1,1,\hdots 1)$. Further, $\tau$ factors through Frobenius precisely when $n > n_1$, and factors through Verschiebung when $n_g > 0$. The Hecke operator $\tau$ induces a closed subvariety $\Agr^{\ord}[\tau]\subset \Agr^{\ord}\times \Agr^{\ord}$, with the two canonical projections $\Pr_1(\tau),\Pr_2(\tau)$ to $\Ag$. The maps $\Pr_i(\tau)$ are usually inseparable.

Now, let $C \subset \Agr$ denote a reduced generically ordinary curve, and let $C^{\ord}$ be the intersection of $C$ with the ordinary locus of $\Agr$. We define $\tau(C^{\ord}) \subset \Agr^{\ord}$ to be the unique reduced subscheme corresponding to $\Pr_1(\tau)_*\Pr_2(\tau)^*C^{\ord}$, and we define $\tau(C)$ to be the Zariski closure of $\tau(C^{\ord})$ in $\Agr$. We make the following observation that will be used below: Let $\tau$ denote a Hecke operator which doesn't factor through Frobenius, and suppose the Serre-Tate coordinates of some branch of the completion of $C$ at an ordinary point are $Q=\{q_{ij}\}$, where none of the $q_{ij}$ are $p$th powers. Then the Serre-Tate coordinates of $\tau$ applied to this branch are calculated using the recipe outlined in Proposition \ref{prop:coordinatesafterisog} (2), i.e. the Serre-Tate coordinates\footnote{This is with respect to appropriate bases, of course. Canonically, given an isogeny, $X,Y$ are the matrices of the induced maps between the corresponding Tate modules and their duals. It is only after we pick compatible $\Z_p$-bases for the Tate modules that the description in terms of matrices as in \ref{prop:coordinatesafterisog} (2) applies.} of the various branches equal ${X^T}^{-1}Q^{p^n}Y$, where the matrix of the isogeny is $(X,Y)$, $X$ is as in \ref{hecketype} and $YX = p^n \Id$.

\subsection{$p$-power Hecke operators on Serre-Tate coordinates}

We will now describe the action of $p$-power Hecke operators locally in terms of Serre-Tate coordinates. For a point $x\in \Agr$, we will denote the associated principally polarized abelian variety by $A_x$ and its associated $p$-divisible group by either $A_x[p^{\infty}]$ or $\pdiv_x$. We note that the formal neighbourhood $\widehat{\Agr}^x$ of $\Agr$ at $x$ is canonically isomorphic to the polarized deformation space of $A_x$ (which, by the Serre-Tate lifting theorem, is the same as the polarized deformation space of $\pdiv_x$). Therefore, this space canonically has the structure of a formal torus, and in coordinates, is isomorphic to $\Hom^{\Sym}(T_p(\pdiv_x)\otimes T_p(\pdiv_x^\vee),\Gmhat)$.
\begin{defn}\label{defnserretatedirection}
Let $x\in \Agr(\Fpbar)$ denote an ordinary point, and let $\ell \subset T_x\Agr$ denote a one-dimensional subspace. We say that $\ell$ is a \emph{Serre-Tate} direction if it equals the tangent space of a rank-1 formal subtorus of $\widehat{\Agr}^x$, i.e. $L\otimes \Gmhat$ where $L\subset (T_p(\pdiv)^*\otimes T_p(\pdiv^\vee)^*)^{\Sym}$ is a saturated rank-1 $\Z_p$-submodule. We say that $\ell$ is a \emph{primitive Serre-Tate} direction if in addition $L$ is spanned by a primitive tensor in $(T_p(\pdiv)^*\otimes T_p(\pdiv^\vee)^*)^{\Sym}$. We remark that primitive tensors of $(T_p(\pdiv)^*\otimes T_p(\pdiv^\vee)^*)^{\Sym}$ are of the form $\epsilon \otimes \mu$ where $\mu = \lambda(\epsilon)$. 

Analogously, we define $T_x^m\Agr$ to be the $m$th order neighborhood of $x$ in $\Agr$. We define a primitive $m$th order Serre-Tate direction to be the restriction of a rank-1 formal subtorus of $\widehat{\Agr}^x$ to this neighborhood.

\end{defn}
We note that for any positive $m$, there are only finitely many $m$th order Serre-Tate directions.

\subsection{Hypersymmetric points}
We now define the notion of hypersymmetric points. As will be clear, we make crucial use of such points to treat $p$-power Hecke operators. 
\begin{defn}\label{hypersymmetricabvars}
An ordinary abelian variety $A/\F_q$ is said to be  \textbf{hypersymmetric} if $A$ is isogenous to $E^g$, where $E/\F_q$ is an ordinary elliptic curve. We say that $x\in \Agr^{\ord}(\F_q)$ is a hypersymmetric point if $A_x$ is a hypersymmetric abelian variety. 
\end{defn}
\begin{lemma}\label{lemma:hypersymmetric}
Let $x\in \Agr(\F_q)$ be ordinary and hypersymmetric, and let $\tau$ be a $p$-power Hecke operator. Then, $\tau(x) \subset \Agr(\F_q)$. Further, every Serre-Tate direction at $x$ thought of as an $\Fpbar$-point of $\Agr$ is defined over $\F_q$.
\end{lemma}
\begin{proof}
Every subgroup of $A_x[p^\infty]$ has the form $G^{\et}\times G^{\mult}$, where $G^{\et} \subset A_x[p^{\infty}]^{\et}$, and $G^{\mult}\subset A_x[p^{\infty}]^{\mult}$. Since $A_x$ is hypersymmetric, the Galois representation on $T_p(A_x[p^{\infty}]^{\et})$ acts through a scalar (because this statement is trivially true for $E$). It follows that every subgroup of $A_x[p^{\infty}]$ is defined over $\F_q$, whence the first part follows. 

As to the $\F_q$-rationality of Serre-Tate directions at $x$, note that the deformation space of $\pdiv_x/\F_q$ is a formal torus isomorphic to $\Hom(T_p(\pdiv)\otimes T_p(\pdiv^{\vee}),\Gmhat)$ and the polarized deformation space is isomorphic to $\Hom^{\Sym}(T_p(\pdiv) \otimes T_p(\pdiv^{\vee}),\Gmhat)$ (note that the Galois action on $T_p(\pdiv)$ and $T_p(\pdiv^{\vee})$ will not be trivial, and so these formal tori need not split over $\F_q$). As $\Gal(\Fpbar/\F_q)$ acts on $T_p(\pdiv)\otimes T_p(\pdiv^{\vee})$ through scalars, every sub-torus of this formal torus defined over $\Fpbar$ descends to $\F_q$. Therefore, every (primitive) Serre-Tate direction over $\Fpbar$ is defined over $\F_q$. We emphasize that the rank-1 formal subtorus $L\otimes \Gmhat$ in Definition \ref{defnserretatedirection} need not be split over $\F_q$. Nonetheless, $L\otimes \Gmhat$ will be a sub-torus defined over $\F_q$. 
\end{proof}


\section{Proofs of the main theorems}
Let $X\subset \Agr$ denote a subvariety. Recall that we say that $X$ is \emph{generically ordinary} if every irreducible component of $X$ has non-empty intersection with $\Agr^{\ord}$. We also say that two points of $\Agr$ are isogenous if the underlying abelian varieties are.

Now, let $C\subset \Agr$ denote a generically ordinary reduced irreducible curve, that has the following properties:
\begin{enumerate}\label{dataC}
    \item $C$ passes through every $\F_q$-rational ordinary hypersymmetric point $x$.\label{dataC1}
    
    \item For any ordinary hypersymmetric point $x \in \Agr(\F_q)$, $C$ is singular at $x$, and for every $m$th order primitive Serre-Tate direction $\ell_m$ at $x$, there exists a formally smooth local branch $C_{\ell_m,x}$ of $C$ at $x$ which restricts to $\ell_m$. \label{dataC2}
    
\end{enumerate}

We have the following key theorem: 
\begin{theorem}\label{thm:ppowerheckesmooth}
Let $\tau$ denote any $p$-power Hecke operator, and let $C$ be as above. Then $\tau(C)$ contains every ordinary $x \in \Agr(\F_q)$ which is hypersymmetric. Further, $\tau(C)$ contains every $m$th order primitive Serre-Tate direction in $T_x^m\Agr$.
\end{theorem}

\begin{proof}
Let $x\in \Agr(\F_q)$ denote some fixed hypersymmetric point, and let $\ell_m$ denote an $m$th order primitive Serre-Tate direction, and let $C_{x,\ell_m} = \Spf \F_q[[t]]$ denote the formally smooth local branch of $C$ at $x$ that realizes $\ell_m$. Consider an ordinary hypersymmetric point $y\in \Agr(\F_q)$ and some $m$th order primitive Serre-Tate direction $\ell'_m$ at $y$. By assumption, $C$ passes through $y$ and there exists a formally smooth local branch $C_{\ell'_m,y}$ of $C$ at $y$ which restricts to $\ell'_m$. Therefore, it suffices to find $(y,\ell'_m)$ such that $\tau(y,\ell'_m) = (x,\ell_m)$. 

Let $\tau^{\vee}$ be the Hecke correspondence dual to $\tau$. We note that $\tau^{\vee}$ has type\footnote{Recall that the type of a Hecke correspondence (by definition) is a monotonically decreasing sequence.} $(n:n'_g,n'_{g-1}, \hdots n'_1)$, where $n'_i = n - n_i$. We will pick $y$ to be a point in $\tau^{\vee}(x)$. By Lemma \ref{lemma:hypersymmetric}, it follows that $y$ and every $m$th order Serre-Tate direction at $y$ is defined over $\F_q$. Therefore, we will work over $\Fpbar$ and treat $\pdiv_x$ and $\pdiv_y$ as $p$-divisible groups over $\Fpbar$ (this allows us to use the analysis of $p$-power Hecke operators and Serre-Tate coordinates carried out in the previous section). We fix polarization-compatible bases $\{e_i \}$ of $T_p(\pdiv_{x})$ and $\{m_j \}$ of $T_p(\pdiv^{\vee}_{x})$, with dual bases $\{\epsilon_i \} $ and $\{ \mu_j\}$, such that $\ell_m$ is realized by the sub-torus corresponding to the span of $\epsilon_1\otimes \mu_1$. 

Let $y\in \tau^{\vee}(x)$, such that the associated isogeny $\phi^{\vee}:A_x \rightarrow A_y$ is obtained by quotienting $A_x$ by the diagonal\footnote{The notion of diagonal is with respect to the ordered bases $\{e_i\}$ and $\{m_j\}$ above.} subgroup with type $(n:n'_1,\hdots n'_g)$. Let $e'_i \in T_p(\pdiv_{y,\Fpbar})$ be defined as $e'_i = \frac{1}{p^{n'_i}}\phi^\vee(e_i)$, and let $m'_j \in T_p(\pdiv_{y,\Fpbar}^{\vee})$ be defined analogously -- $\{e'_i\}$ and $\{ m'_j \}$ are bases of $T_p(A_{y,\Fpbar})$ and $T_p(A_{y,\Fpbar}^{\vee})$ respectively. Let $\{\epsilon'_i \}$ and $\{ \mu'_j\}$ denote dual bases. Consider the branch $C_{y,\ell'_m}$, where $\ell'_m$ is realized by $\epsilon_1' \otimes \mu_1'$. The Serre-Tate coordinates of $C_{y,\ell'_m}$ are given by functions $q'_{ij}(t)$ where $q'_{11}(t) \equiv 1+t \mod t^{m+1}$ and $q'_{ij}(t) \equiv 1 \mod t^{m+1}$. 

Let $\phi:A_y\rightarrow A_x$ denote the dual of $\phi^{\vee}$. We now lift $\phi$ to an isogeny $\tilde{\phi}$ whose source is the abelian scheme $A_{C_{y,\ell'_m}}$. By Proposition \ref{prop:coordinatesafterisog}, the Serre-Tate coordinates of $\liftphi(A_{C_{y,\ell'_m}})$ are given by $Q(s) = X^{-1}Q'(s^{p^n})Y$, where $X$ and $Y$ are diagonal matrices with $i$th entry $p^{n_i}$ and $p^{n-n_i}$ respectively. It follows that the $m$th order tangent direction of $\liftphi(A_{C_{y,\ell'_m}})$ is equal to $\ell_m$, and therefore that $(x,\ell_m) \in \tau(y,\ell'_m)$. The result follows. 
\end{proof}

Now, let $C\subset \Agr$ be a generically ordinary curve whose geometric fundamental group surjects onto $\Gamma_g(\Z_r)\times\prod_{\ell\neq p,r} \bbP(\Sp_{2g}(\Z_\ell))\times \GL_g(\Z_p)$. 
\begin{proposition}\label{prop:heckeirreducible}
    Let $C$ be as above, and let $\tau$ denote either a $p$-power or a prime-to-$p$ Hecke operator. Then, $\tau(C)$ is geometrically irreducible.
\end{proposition}
\begin{proof}
    The proof is essentially the same in both cases, and the key ingredient is that the monodromy of $C$ equals the monodromy of $\Agr$. We will assume that $\tau$ is $p$-power without any loss of generality. Recall that $\tau$ is defined on $\Agr^{\ord}$, and $\tau(C)$ is defined to be the Zariski closure of $\tau(C^{\ord})$ in $\Agr$. Therefore, it suffices to prove that $\tau(C^{\ord})$ is irreducible. Recall that $\Pr_i(\tau)$ are maps from $\Agr^{\ord}[\tau]$ to $\Agr^{\ord}$. As $\tau$ is defined by $\Pr_1(\tau)_* \Pr_2(\tau)^*$, it suffices to prove that $\Pr_2^*C^{\ord}$ is irreducible. Let $\Ag[p^nr]^{\ord}$ be the level cover of $\Agr^{\ord}$ where the $p$-adic Tate module has been trivialized mod $p^n$. Up to pre-composing by a power of Frobenius \footnote{This is required as the map $\Pr_2$ usually isn't separable.}, the canonical map $\Ag[p^nr]^{\ord} \rightarrow \Agr$ factors through $\Agr^{\ord}[\tau]$. The pullback of $C$ to $\Ag[p^nr]^{\ord}$ is irreducible as $C$ has maximal monodromy. The result follows. 
\end{proof}

\subsection{Zariski-Density of rank-1 formal branches}

\begin{proposition}\label{notalltori}

Let $x \in \Agr(\Fpbar)$ be an ordinary point, and suppose $D\subset \Agr$ is a variety
such that the formal neighbourhood of $D$ at $x$ contains all rank-1 primitive subtori. Then $D=\Agr$. 

\end{proposition}

\begin{proof}

We will in fact prove the stronger claim that all primitive rank-1 formal tori are dense in $\Sym^{2}(\Z_p^g)\otimes\Gmhat$. Letting $\{q_{ij}\}$ be the co-ordinates as before, we consider the element $v\in\Z_p^g$ such that $v_i=p^{N^i}$ where $N$ will be chosen to be a large integer. Then the corresponding primitive rank 1 torus has co-ordinates $q_{ij}(t)=t^{p^{N^i+N^j}}$. Note that $\ds \prod_{i,j}q_{ij}^{r_{ij}}$ evaluates at $\vec{Q}(t)$ to  $t^{M}$ where $M_{ij}=\ds\sum_{i,j}r_{ij}p^{N^i+N^j}$. 

Now let $f\in k[[q_{ij}]]$ be a non-zero power series which vanishes on all such primitive rank 1 tori. By choosing $N$ to be large enough, there is a single term of $f(\vec{Q}(t))$ of smallest degree. Therefore $f(\vec{Q}(t)) \neq 0$, which is a contradiction. The result follows. 


\end{proof}

\subsection{Algebraicity of primitive Serre-Tate directions}

Recall that $D \subsetneq \Agr$ is a generically ordinary subvariety. As in \cite[Prop 5.4]{Chai}, we may consider the completion along the diagonal in $D\times D$ as a formal scheme $D^\wedge$ over $D$, and we may also view this as sitting inside the pull back $i^*\Agr^\wedge$, where $i: D\rightarrow \Agr$ is the inclusion. For any integer $m>0$ we may also take the corresponding $m$th order infinitesimal subscheme $D^{\wedge}_m$ over $D$ sitting inside $i^*{\Agr^{\wedge}}_{m}$. Consider the level cover ${\Ag[p^mr]}^{\ord}$ which trivializes the $p$-adic Tate module mod $p^m$, and let $(\Agr^{\wedge})_m/\Ag[p^mr]^{\ord}$ denote the base change of $(\Agr^{\wedge})_{m}$ to $\Ag[p^mr]^{\ord}$. 
We have that $(\Agr^{\wedge})_m/\Ag[p^mr]^{\ord}$ can be trivialized as $\Sym^2(\Z/p^m\Z)^g\otimes {\Gmhat} \times \Ag[p^mr]^{\ord}$. The space of primitive Serre-Tate directions is a closed subscheme of $(\Agr^{\wedge})_m/\Ag[p^mr]^{\ord}$, which descends to a closed subscheme of $(\Agr^{\wedge})_m/\Agr^{\ord}$. Let $R_m$ denote the restriction of this closed subscheme to $D$ -- its fibers at closed points consist of the unions of all primitive Serre-Tate directions to order $m$. Both $R_m$ and $D^{\wedge}_m$ are subschemes of $i^*(\Agr^{\wedge}_m)$. 
It follows that there is a Zariski-closed subscheme of $D$ consisting of all the points at which $D$ contains all the primitive Serre-Tate directions to order $m$. We therefore have the result below that follows directly from the above analysis, Proposition \ref{notalltori}, and the Noetherian property. 
\begin{proposition}\label{prop: Dstempty}
 Let $D\subsetneq \Agr$ be a generically ordinary subvariety. There exists a positive integer $N$ such that for every ordinary point $x$, there exists an order $N$ primitive Serre-Tate direction of $x$ that isn't contained in $D$.
\end{proposition}

\subsection{Proof of Theorem \ref{thm:main}}

Suppose that $D\subsetneq \Agr$ is a subvariety, and let $N$ be as in Proposition \ref{prop: Dstempty}. Let $x\in\Ag(\F_q)$ be a hypersymmetric point. 

We increase $\F_q$ so that the pre-images of $\pi(x)$ in $\Ag[r]$ are defined over $\F_q$, where $\pi$ is the canonical map from $\Agr$ to $\Ag$. We use Theorem \ref{congmain} to produce a curve $C\subset\Ag[r]$ which satisfies the following.
\begin{enumerate}
    \item For every hypersymmetric point over $\F_q$ isogenous to $x$, and for every $N$th order primitive Serre-Tate direction $\eta$ through $x$, $C$ admits a $\F_q[[t]]$ point specializing to $\eta$.
    \item The geometric fundamental group of $C$ surjects onto $\Gamma_g(\Z_r)\times\prod_{\ell\neq p,r} \bbP(\Sp_{2g}(\Z_\ell))\times \GL_g(\Z_p)$.
\end{enumerate}

Let $C'\subset \Agr$ denote an irreducible curve with the property $A/k(C')$ is geometrically isogenous to $A/k(C)$ (where $A/k(C)$ and $A/k(C')$ are the universal abelian scheme pulled back to the function fields of $C,C'$ respectively). We then have that $C'$ must be an irreducible component of some Hecke translate of $C$. By condition 2 above, every Hecke translate of $C$ is irreducible. Therefore, it suffices to show that $\tau(C)\not\subset D$ for every Hecke operator $\tau$. We factor $\tau=\tau_e\tau_p$ for $\tau_e$ a prime-to-$p$ Hecke operator and $\tau_p$ a $p$-power Hecke operator.

By Theorem \ref{thm:ppowerheckesmooth}, the curve $\tau_p(C)$ also satisfies condition 1 above - in particular, $x$ is also a point of $\tau_p(C)$. Now since $\tau(C)$ is irreducible, it follows that for each point $y\in \tau_e x$ the formal neighborhood of $\tau(C)=\tau_e(\tau_p(C))$ at $y$ contains the image of the formal neighborhood of $\tau_p(C)$ at $x$. Since $\tau_e$ is \'etale, for every $N$th order primitive Serre-Tate direction $\eta'$ through $y$, we have that $\tau(C)$ admits an $\F[[t]]$ point specializing to $\eta'$. Therefore, $\tau(C)$ cannot be contained in $D$ by Proposition \ref{prop: Dstempty}. 

The result about function field valued points follows immediately by considering the generic point of $C$. We have therefore proved the existence of a Hodge-generic point of $\Agr$ which does not admit a \emph{polarized} isogeny to any point of $D$. The statement about arbitrary isogenies follows from the fact that any isogeny between Hodge-generic principally polarized abelian varieties must be a polarized isogeny. This finishes the proof of Theorem \ref{thm:main}.

$\medskip\square$

\subsection{Proof of Theorem 1.1}

We now suppose that $D\subsetneq (\Agr)_{F}$ is a subvariety. By replacing $D$ by its Galois translates, we may assume that $D$ is defined over $\F_q(T)$. 
Identifying $\F_q(T)$ with the function field of $\bbP^1$ we may realize $D$ as the generic fiber\footnote{That is, we may treat $\cD$ as being fibered over $\bbP^1$, with the property that $D$ is the generic fiber of $\cD$.} of a (strict) subvariety $\cD\subset (\Agr\times\bbP^1)_{\F_q}$, no generic points of which are contained in a fiber. We let $E\subset(\Agr)_{\Fpbar}$ be the constant part $D$, so that $E_F\subset D$ and $E$ is the maximum such subvariety\footnote{One may construct $E$ by intersecting all base changes of $D$ along automorphisms of $F$ over $\Fpbar$.}. By increasing $q$ we assume $E$ is defined over $\F_q$. We are looking for a curve $C\subset (\Agr\times\bbP^1)_{\F_q}$ such that $\tau(C) \not\subset \cD$ for every Hecke operator $\tau$. We shall need the following two lemmas:

\begin{lemma}\label{galois}
Let  $K\subset L$ be a finite separable extension of fields, and let $G$ be the Galois group of the normal closure of $L$ over $K$. Suppose that $H$ is a profinite group with no irreducible constituent in common with $G$, and let $\phi:\Gal_K\ra H$ be a group homomorphism. Then $\phi(\Gal_L)=\phi(\Gal_K)$.
\end{lemma}

\begin{proof}

By increasing $L$ we may as well assume $L/K$ is Galois. Then $\Gal_L\subset\Gal_K$ is normal, and so 
we get a map $G\ra\phi(\Gal_K)/\phi(\Gal_L)$, which must be trivial by our assumptions. The statement follows.
\end{proof}

\begin{lemma}\label{uniform}

Let $\cD\subset (\Agr\times\bbP^1)_{\F_q}$ be as above. There is an integer $M$ such that for any $M$ distinct points $x_1,\dots,x_M$ of $\bbP^1(\Fpbar)$, we have $\ds\bigcap_{i=1}^M \cD_{x_i}=E$.

\end{lemma}

\begin{proof}

We will base change to an uncountable field $k$, and prove the statement there (from which the original statement clearly follows).

For any $M$, consider the set $S_M$ of points $x_1$ for which one can find distinct $x_2,\dots,x_M$ such that $\ds\bigcap_{i=1}^M \cD_{x_i}\neq E$. This is clearly a constructible set, so is either finite or co-finite. Moreover the set $S_M$ is clearly descending (perhaps non-strictly) with $M$. It follows that if $S_M$ is not eventually empty, then $\bigcap_M S_M\neq\emptyset.$ Thus we may find a point $x_1$ which is contained in this intersection.

We may now iterate with $x_2$, etc, to produce a countably infinite sequence of points $x_i$ such that 
$\ds\bigcap_{i=1}^m \cD_{x_i}\neq E$ for every positive integer $m$. By the Noetherian property, this intersection eventually stabilizes to a subvariety $E'$ properly containing $E$. This $E'$ is contained in infinitely many fibers of $\cD$ and hence in all of them. Thus $E'_F\subset\cD$ contradicting the definition of $E$.

\end{proof}

We now construct our curve $C$ as follows. Following Theorem \ref{thm:main} we construct an irreducible curve $C_0\subset \Ag[r]$ such that
\begin{enumerate}
    \item No Hecke translate of $C_0$ is contained in $E$.
    \item The geometric fundamental group of $C_0\cap\Ag^{\ord}[r]$ surjects onto $\Gamma_g(\Z_r)\times\prod_{\ell\neq p,r} \bbP(\Sp_{2g}(\Z_\ell))\times \GL_g(\Z_p)$.
\end{enumerate}

We pick $N$ sufficiently large so that the alternating group $A_N$ is not a sub-quotient of $\prod_{\ell\neq p} \bbP(\Sp_{2g}(\Z_\ell))\times \GL_g(\Z_p)$, and let $\phi:\bbP^1\ra\bbP^1$ be a degree $N$ map whose Galois group is $A_N$. Next, we pick a map $\psi:C_0\ra\bbP^1$ arbitrarily, and let $C_1$ be the graph of this map in $\Ag[r]\times\bbP^1$. Finally, we let $C$ be the base change of this map under $\Id_{\Ag[r]}\times\phi$, for $N>\deg\psi$. It follows that $C$ is irreducible since $A_N$ does not admit a non-trivial homomorphism to $S_{\deg\psi}$. 

By Lemma \ref{galois}, the geometric fundamental group of $C$ surjects onto $\Gamma_g(\Z_r)\times\prod_{\ell\neq p,r} \bbP(\Sp_{2g}(\Z_\ell))\times \GL_g(\Z_p)$. It follows that $\tau(C)$ is irreducible for each Hecke operator $\tau$. Suppose that $\tau(C)\subset \cD$. For a generic point $s\in\bbP^1$, the fiber $\phi^{-1}(s)$ consists of $N$ points $x_1,\dots,x_N$, and so it follows from Lemma \ref{uniform} that $$(\tau(C_1))_{s}=\bigcap_{i=1}^N (\tau(C))_{x_i}\subset\bigcap_{i=1}^N \cD_{x_i}=E.$$ It follows that $\tau(C_1)\subset E$. Since $E$ is constant, by projecting to $\Agr$, this in fact means that $\tau(C_0)\subset E$ which is a contradiction.

\section{Additional results}
\subsection{The case of number fields}
We work now over $\Qbar$, so in particular in characteristic 0. Let $H$ denote the Hodge line bundle on $\cA_g[r]$. For any proper curve $C\subset \cA_g[r]$, we define $(C.H)$ to denote the degree of $H$ pulled back to $C$.


\begin{lemma}\label{ellheckedegree}
Let $\tau$ denote a Hecke operator on $\cA_g[r]$. Then, for any proper curve $C\subset \cA_g[r]$ we have $\deg \tau(C) = \deg(\tau) \cdot \deg(C)$.  
\end{lemma}
\begin{proof}
Let $\Agr[\tau] \subset \Agr \times \Agr$ denote graph of $\tau$, and let $\Pr_i[\tau]$ denote the two projections restricted to $\Agr[\tau]$. It suffices to prove that $\Pr_1[\tau]^*H$ is isomorphic to $\Pr_2[\tau]^*H$. 

Let $\phi: A \rightarrow B$ denote the universal isogeny on $\Agr[\tau]$. The pullback of $\wedge^g\Omega^1_{A/\Agr[\tau]}$ via the identity section is isomorphic to $\Pr_1[\tau]^*H$, and the pullback of $\wedge^g \Omega^1_{B/\Agr[\tau]}$ via the identity section is isomorphic to $\Pr_2[\tau]^*H$. Therefore, it suffices to prove that $\phi^* \Omega^1_{B/\Agr[\tau]} \simeq \Omega^1_{A/\Agr[\tau]}$. This is true because $\phi$ is \'etale, as it is an isogeny in characteristic 0.
\end{proof}

\begin{defn}
Let $X$ be some variety defined over a field $L$, and let $K/L$ denote a finite extension. We say that a point $x\in X(K)$ is \emph{primitive} if $K$ is the smallest field of definition of $x$.
\end{defn}

\begin{lemma}\label{pointlargedegreelargemonodromy}
Let $C\subset \cA_g[r]$ denote a curve defined over $\Q$ with maximal monodromy. There exists a constant $c \in \Z$ such that for any positive integer $n\in \Z$, there exists a number field $K/\Q$ which is Galois and has degree $\geq n$, and a primitive point $x\in C(K)$ such that the monodromy of $A_x$ has index bounded by $c$ in $\GSp_{2g}(\hat{\Z})$.
\end{lemma}
\begin{proof}
Let $\pi: C\rightarrow \Pone$ denote a finite map of degree $d$, and consider $B = \Res_{C/\Pone}(A)$, which is an abelian scheme over an open subset $U\subset \Pone$ of dimension $gd$. Work of Zywina \cite[Theorem 1.1]{Zywina} implies that for any number field $L$, there exists a constant $c_L$ such that for a density 1 set of points $y\in U(L)$, the monodromy of $B_y$ has index bounded by $c_L$ in the monodromy of $B_U$. In Section 1.2 of \emph{loc. cit.}, Zywina explicitly describes the constant $c_L$. Indeed, if $L/\Q$ were a Galois extension which has the property that the monodromy of $B_{U_L}$ were the same as the monodromy of $B_{U_{\Q}}$, then the constant $c_L$ equals $c_{\Q}$. It is easy to see that there exist fields $L$ with this property with $[L:\Q]$ arbitrarily large; for example, by choosing $L$ so that $\Gal(L/\Q) = \textrm{PSL}_N(\F_\ell)$, where $N>2\dim B$ and $\ell$ is any sufficiently large prime.

We now use the above discussion to deduce the existence of number fields $K$ with arbitrarily large degree and points $x\in C(K)$ with monodromy as claimed. Let $y\in U(L)$ satisfy the conclusions of the previous paragraph. By Hilbert irreducibility, we may also assume that that $\pi^{-1}(y) = \{x_1\hdots x_d \}$, where $x_i\in C(K_i)$, where $K_i/L$ are conjugate degree $d$ extensions of $L$ and the points $x_i$ are conjugate to each other. We have that $\prod_{i=1}^d A_{x_i}$ descends to $L$ and equals  $ B_y$, and that the monodromy of $(\prod_{i=1}^d A_{x_i})_{K_1}$ equals the intersections of the monodromies of $B_C$ and $B_y$ inside the monodromy of $B_U$. It then follows that the monodromy of $(\prod_{i=1}^d A_{x_i})_{K_1}$ lies in the monodromy of $B_y$ and is the subgroup fixing the Tate-module of the first factor $A_{x_1}$. 
It follows that the monodromy group of $(\prod_{i=1}^d A_{x_i})_{K_1}$ has index bounded by $c = c_{\Q}$ inside the monodromy of $B_C$, and projecting onto the first factor yields that the monodromy of $A_{x_1}$ has index bounded by $c$ in the monodromy of $A_C$, which is equal to $\GSp_{2g}(\hat{\Z})$ by our assumption on the curve $C$. The lemma follows by taking $K = K_1$ and $x = x_1$.

\end{proof}

\begin{theorem}\label{mainnumberfield}
Let $V\subsetneq \cA_g[r]$ denote a closed subvariety over $\overline{\Q}$. Then, there exists an abelian variety over a number field which is not isogenous to any point of $V$. 
\end{theorem}
\begin{proof}
We first note that the result is trivial for $g = 1$, so we will assume throughout that $g>1$. Without loss of generality, we suppose that $V$ is a divisor, and by replacing $V$ by the union of its Galois conjugates, we may assume that $V$ is defined over $\Q$. Moreover, by increasing $V$ we may assume that $\cO(V)$ is a power of the Hodge bundle $H$. This implies that $(C\cdot V) >0$ for every proper curve $C\subset \Agr$. 

We now proceed as in the positive characteristic setting (but with a much easier argument) to find a curve $C\subset \cA_g[r]$ having maximal monodromy which satisfies $\tau(C) \not\subset V$ for every Hecke operator $\tau$. By a Noetherianity argument, we first find a positive integer $k$ such that for all points $x\in\cA_g[r]$, $V$ does not contain $\Spec\cO_{\cA_g[r],x}/(\mathfrak{m}_x^k)$. Next, by using the characteristic-zero Lefschetz theorem detailed in \cite[Theorem 1.2]{GorMacMorse} for smooth quasi-projective varieties, and using the fact that the boundary is codimension $>1$ in the Bailey-Borel compactification (recall that $g > 1$), we may find a proper irreducible curve $C\subset \cA_g[r]$ defined over $\Q$ with maximal monodromy such that $C$ does contain $\Spec\cO_{\cA_g[r],x}/(\mathfrak{m}_x^k)$ for some $x$. It then follows that $\tau(C)$ is irreducible for every Hecke operator $\tau$, and since $\tau$ is \'etale $\tau(C)$ cannot be contained in $V$.

Next, by Lemma \ref{ellheckedegree} we have \begin{equation}\label{equationupperbound}
(\tau(C) \cdot V) = \deg(\tau) (C \cdot V).    
\end{equation}
By Lemma \ref{pointlargedegreelargemonodromy}, there exists a Galois extension $K/Q$ of degree $\geq n$ and a primitive point $x\in C(K)$ whose monodromy has index bounded by $c$ in $\GSp_{2g}(\hat{\Z})$, where we now choose $n > c (C\cdot V).$ For any fixed Hecke operator $\tau$, we have that $\tau(x)$ is stabilized by $\Gal(\overline{K}/K)$, and the monodromy assumption on $A_x$ implies that each orbit has at least $\frac{\deg(\tau)}{c}$ elements. The same statement holds for $\sigma(x)$ with $\sigma \in \Gal(K/\Q)$, and further, we have that the sets $\tau(x)$ and $\tau(\sigma(x))$ are in the same $\Gal(\overline{\Q}/\Q)$ orbit. 

Suppose now that $\tau(x)$ intersected $V$ non-trivially for some Hecke operator $\tau$. As we have assumed $V$ is defined over $\Q$, it follows that $\tau(x) \cap V$ is stable under $\Gal(\overline{\Q}/\Q)$.

 The previous paragraph implies that $\tau(\sigma(x)) \cap V$ has size $\geq \frac{\deg(\tau)}{c}$ for every $\sigma\in \Gal(K/\Q)$. Further, even if the same point $y$ is contained\footnote{This can happen if $\sigma_1(x)$ is isogenous to $\sigma_2(x)$.} in $\tau(\sigma_1(x)) \cap V$ and $\tau(\sigma_2(x)) \cap V$ for $\sigma_1\neq\sigma_2$, this point contributes twice to the intersection $\tau(C) \cdot V$, because this implies that $\tau(C)$ must be singular at $y$, with branches corresponding to the branch of $C$ through $\sigma_1(x)$ and through $\sigma_2(x)$. It therefore follows that 
 \begin{equation}\label{equationlowerbound}
 (\tau(C)\cdot V) \geq \frac{\deg(\tau)}{c}\cdot n.
 \end{equation}
 But our initial choice of $n$ was such that $n> c (C\cdot V)$, and so Equations \eqref{equationupperbound} and \eqref{equationlowerbound} cannot both hold simultaneously. It therefore follows that $\tau(x)$ must be disjoint from $V$ for every Hecke operator $\tau$, and the result follows.

\end{proof}

\subsection{An analogue over $\Fpbar$.}
We will now prove that for $D\subset \cA_g[r]$ a divisor, there exists an ordinary $x\in \cA_g[r](\Fpbar)$ such that $\tau(x) \not\subset D$ for every prime-to-$p$ Hecke operator $\tau$. 

\begin{proof}[Proof of Theorem \ref{finitefields}]
By the main result of this paper, there exists a generically ordinary $C\subset \cA_g[r]$ that satisfies $\tau(C)\not\subset D$ for every Hecke operator $T$. An identical intersection-theoretic argument as in the proof of Theorem \ref{mainnumberfield} yields that for a point $x\in C(\Fpbar)$, whose minimal field of definition is $\F_q$ where $q$ is a large enough power of $p$, $\tau(x)\not\subset D$. The result follows. 

\end{proof}

\bibliographystyle{alpha}

\end{document}